\documentclass{amsart}
\usepackage{amssymb}
\usepackage{amsmath}
\usepackage{amscd}
\usepackage{amsthm}
\usepackage[centertags]{amsmath}
\usepackage{amsfonts}
\usepackage{newlfont}
\usepackage{graphicx}
\usepackage[usenames]{color}
\usepackage{amsfonts, amssymb}
\usepackage{mathrsfs}
\usepackage{latexsym}
\usepackage{nicefrac}

\usepackage{verbatim}
\usepackage{tabulary,tabularx}
\usepackage[all]{xy}
\usepackage{pgfplots}
\usepackage{tikz}
\usetikzlibrary{backgrounds}
\usetikzlibrary{intersections}
\pgfdeclarelayer{background}
\pgfdeclarelayer{foreground}
\pgfsetlayers{background,main,foreground}
\providecommand{\abs}[1]{\lvert#1\rvert}
\providecommand{\norm}[1]{\lVert#1\rVert}

\newtheorem{theorem}{Theorem}[section]

\newtheorem{proposition}[theorem]{Proposition}

\theoremstyle{definition}
\newtheorem{remark}[theorem]{Remark}

\theoremstyle{definition}

\makeatletter
\@namedef{subjclassname@2020}{\textup{2020} Mathematics Subject Classification}
\makeatother

\theoremstyle{remark}
\def\X{\mathcal{X}}

\newcommand{\R}{\mathbb R}

\newcommand{\vol}{\operatorname{vol}}
\newcommand{\diam}{\operatorname{diam}}

\newcommand{\conv}{\operatorname{conv}}

\newcommand{\tr}{\operatorname{tr}}
\newcommand{\Id}{\operatorname{Id}}

\newcommand{\ls}{\leqslant}
\newcommand{\gr}{\geqslant}
\renewcommand{\epsilon}{\varepsilon}

\begin{document}
\title[Continuous quantitative Helly-type results]{Continuous quantitative Helly-type results}


\author{Tom\'as Fernandez Vidal}
\author{Daniel Galicer}
\author{Mariano Merzbacher}

\address{ Departamento de Matem\'{a}tica - IMAS-CONICET,
Facultad de Cs. Exactas y Naturales  Pab. I, Universidad de Buenos Aires
(1428) Buenos Aires, Argentina}
\email{tfvidal@dm.uba.ar}
\email{dgalicer@dm.uba.ar} \email{mmerzbacher@dm.uba.ar} 


\begin{abstract}

Brazitikos' results on quantititative Helly-type theorems (for the volume and for the diameter) rely on the work  of Srivastava on sparsification of John's decompositions. We change this technique by a stronger recent result due to Friedland and Youssef. This, together with an appropriate selection in the accuracy of the approximation, allow us to obtain Helly-type versions which are sensitive to the number of convex sets involved.
\end{abstract}

\maketitle

\section{Introduction}

Helly's classical result \cite{helly1923mengen} states that if $\mathcal C = \{ C_i : i \in I \}$ is a finite family of at least $n+1$ convex sets in $\R^n$ and if any $n+1$ members of $\mathcal C$ have non-empty intersection then $\bigcap_{i \in I} C_i$ is non-empty.
In general, a Helly-type property  is a property  $\Pi$ for which there exists a number $s\in {\mathbb N}$ such that if $\{ C_i:i\in I\}$ is a finite family of certain objects and every subfamily of $s$ elements fulfills $\Pi $, then the whole family fulfills  $\Pi $.

In the eighties, B\'{a}r\'{a}ny, Katchalski and Pach proved the following quantitative ``volume version" of Helly's theorem \cite{barany1982quantitative,barany1984helly}:
\bigskip

{\sl
Let ${\mathcal C}=\{ C_i:i\in I\}$ be a finite family of convex sets in ${\mathbb R}^n$. If the intersection of
any $2n$ or fewer members of ${\mathcal H}$ has volume greater than or equal to $1$, then
$\vol(\bigcap_{i\in I} C_i)\geq d(n)$, where $d(n)>0$ is a constant depending only on $n$.
}

\bigskip
Thus, the previous result express the fact that \emph{``the intersection has large volume''} is a Helly-type property for the family of convex sets.

Since every (closed) convex set is the intersection of a family of closed half-spaces; a simple compactness argument (see \cite{barany1982quantitative}) shows that one can remove the restriction that ${\mathcal C}$ is finite and also  assume that each convex set is a closed half-space i.e.,
$$\{x \in \R^n : \langle x, v_i\rangle \leq 1\},$$
for some vector $v_i \in \R^n$.
Therefore, the theorem of B\'{a}r\'{a}ny et. al. is  equivalent to the following statement:
\bigskip

{\sl
Let ${\mathcal H}=\{ H_i:i\in I\}$ be a family of closed half-spaces in ${\mathbb R}^n$
such that $\vol(\bigcap_{i\in I}H_i)=1$. There exist $s\leq 2n$
and $i_1,\ldots ,i_s\in I$ such that
\begin{align*} 
\vol(H_{i_1}\cap \cdots \cap H_{i_s})^{1/n} \leq c(n),
\end{align*}
where $c(n)>0$ is a constant depending only on $n$.
}
\begin{figure}
    \centering
   \begin{tikzpicture}
\begin{scope}[scale=0.5]
 \path[fill=gray!60] (0,0) coordinate(p1) --  ++(31:1.5) coordinate(p2) -- ++(15:2) coordinate(p3) -- ++(-21:1.7) coordinate(p4) -- ++(-31:2) coordinate(p5) -- ++ (-41:1.1) coordinate (p6) -- ++ (-80:2) coordinate(p7) -- 
 ++(-100:2) coordinate(p8) --  ++(-125:1.5) coordinate(p9) -- ++ (190:2) coordinate(p10) -- ++(180:2.1)coordinate(p11) -- ++(170:3) coordinate(p12) -- ++ (145:2) coordinate(p13)-- ++ (110:2) coordinate(p14)-- ++ (56:2) coordinate(p15);
 \foreach \X [count=\Y] in {2,...,16}
 {\ifnum\X=16
   \path (p\Y) -- (p1) coordinate[pos=-0.0](a\Y) coordinate[pos=0](a1)
   coordinate[pos=0.5](m1);
   \draw (p\Y) -- (p1);
  \else
   \path (p\Y) -- (p\X) coordinate[pos=-0.0](a\Y) coordinate[pos=0](a\X)
   coordinate[pos=0.5](m\X);
   \draw (p\Y) -- (p\X);
  \fi}
  
  \end{scope}
  \begin{scope}[opacity=0.5,scale=0.5]
 \path (0,0) coordinate(p1) --  ++(31:1.5) coordinate(p2) -- ++(15:2) coordinate(p3) -- ++(-21:1.7) coordinate(p4) -- ++(-31:2) coordinate(p5) -- ++ (-41:1.1) coordinate (p6) -- ++ (-80:2) coordinate(p7) -- 
 ++(-100:2) coordinate(p8) --  ++(-125:1.5) coordinate(p9) -- ++ (190:2) coordinate(p10) -- ++(180:2.1)coordinate(p11) -- ++(170:3) coordinate(p12) -- ++ (145:2) coordinate(p13)-- ++ (110:2) coordinate(p14)-- ++ (56:2) coordinate(p15);

\path (p1) -- (p2) coordinate[pos=-4.5](a1) coordinate[pos=4.0](a2)
   coordinate[pos=0.5](m2);
 \path (p5) -- (p6) coordinate[pos=-4.7](a5) coordinate[pos=3.2](a6)
   coordinate[pos=0.5](m6);
   \path (p7) -- (p8) coordinate[pos=-1.5](a7) coordinate[pos=2.0](a8)
   coordinate[pos=0.5](m8);
     \path (p9) -- (p10) coordinate[pos=-1.2](a9) coordinate[pos=3.5](a10)
   coordinate[pos=0.5](m10);
     \path (p12) -- (p13) coordinate[pos=-1.7](a12) coordinate[pos=3.0](a13)
   coordinate[pos=0.5](m13);
 \coordinate (i1) at (intersection of a1--a2 and a5--a6);
 \coordinate (i2) at (intersection of a5--a6 and a7--a8);
 \coordinate (i3) at (intersection of a7--a8 and a9--a10);
 \coordinate (i4) at (intersection of a9--a10 and a12--a13);
 \coordinate (i5) at (intersection of a12--a13 and a1--a2);
\coordinate (i6) at (intersection of a13--a14 and a1--a2);
\path[draw,fill=gray!60] (i1)--(i2)--(i3)--(i4)--(i5)--(i1);
\end{scope}
\begin{scope}[scale=0.5]
\node at (-1,-3.5) {$\bigcap\limits_{i\in I} H_i$};
\node at (6,-7.5) {$H_{i_1} \cap\dots \cap H_{i_s}$};
\end{scope}

\end{tikzpicture}
    \caption{A convex body defined as the intersection of half-spaces which is enclosed by a convex set given by the intersection of a few of them.}
    \label{fig:helly}
\end{figure}
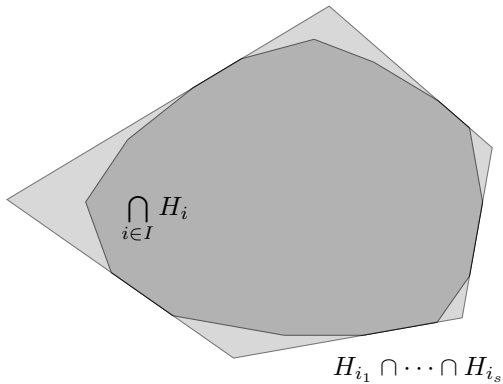

\bigskip

Of course  one cannot replace $2n$ by $2n-1$ in the statement above. Indeed, the cube $[-1/2,1/2]^n$ in ${\mathbb R}^n$ can be written as the intersection of the $2n$ closed half-spaces $$H_j^{\pm }:=\left\{ x:\langle x,\pm \frac{1}{2} e_j\rangle \leq 1 \right\}$$
and that the intersection of any $2n-1$ of these half-spaces has infinite volume.

The authors of \cite{barany1982quantitative} gave the bound $c(n) \leq n^{2n}$ for the constant $c(n)$ and conjectured that one might actually
have polynomial growth i.e., $c(n)\leq n^{d}$ for an absolute constant $d>0$.  Nasz\'{o}di \cite{naszodi2016proof} has verified this conjecture; namely, he proved
that $c(n) \leq c n^{2}$, where $c>0$ is an absolute constant. A clever but slight refinement of Nasz\'{o}di's argument, due to Brazitikos \cite[Theorem 3.1.]{brazitikos2017brascamp}, leads to the exponent $\frac{3}{2}$ instead of $2$.

Moreover, Brazitikos showed in \cite[Theorem 1.4.]{brazitikos2017brascamp} that if we relax the condition on the number $s$ of half-spaces that we use (but still require that it is proportional to the dimension) one can improve significantly the estimate, giving a bound of order $n$.

\begin{theorem}{\cite[Theorem 1.4.]{brazitikos2017brascamp}}\label{teo: volume brazitikos}
 There exists an absolute constant $\alpha>0$ with the following property: for every family ${\mathcal H}=\{ H_i:i\in I\}$ of closed half-spaces in ${\mathbb R}^n$,
$$
H_i = \{x \in \R^n: \langle x, v_i\rangle \leq 1\},
$$
with $\vol(\bigcap _{i\in I}H_i)=1$, there exist $s \leq \alpha n$ and $i_1, \dots, i_s \in I$ such that
\begin{align*}
\vol(H_{i_1}\cap \cdots \cap H_{i_s})^{1/n} \leq cn,
\end{align*}
where $c>1$ is an absolute constant.
\end{theorem}





 
\bigskip
B\'{a}r\'{a}ny, Katchalski and Pach also studied the question whether ``\emph{the intersection has large diameter}" is a sort of Helly-type
property for convex sets. They provided the following quantitative answer to this question:

\smallskip

{\sl Let $\{ C_i:i\in I\}$ be a family of closed convex sets in ${\mathbb R}^n$
such that ${\diam}\left (\bigcap_{i\in I}C_i\right )=1$. There exist $s\ls 2n$ and $i_1,\ldots ,i_s\in I$ such that
\begin{align*}{\diam}\left (C_{i_1}\cap \cdots \cap C_{i_s}\right )\ls (cn)^{n/2},\end{align*}
where $c>0$ is an absolute constant.}

\smallskip

In the same work the authors conjectured that the bound $(cn)^{n/2}$ should be polynomial in $n$. Leaving aside the requirement that $s\ls 2n$, Brazitikos in \cite{brazitikos2017quantitative} provided the following relaxed positive answer:
\begin{theorem}\label{th:diameter}

There exist an absolute constant $\alpha>1$ with the following property: 
for every $n \in \mathbb{N}$ and every  finite family  $\{C_i: i\in I\}$ of convex bodies in ${\mathbb R}^n$ such that $\bigcap_{i\in I}C_i$ has non-empty interior, then there is $z\in  \mathbb{R}^n$, $s\ls \alpha n$ and $i_1,\ldots i_s\in I$ such that

$$ 
z+C_{i_1} \cap \dots \cap C_{i_s} \subset c n^{3/2} \left( z + \bigcap_{i\in I}C_i \right),
$$

where $c>0$ are  absolute constant.

\bigskip
In particular, for every $n \in \mathbb{N}$ and every  finite family  $\{C_i: i\in I\}$ of convex bodies in ${\mathbb R}^n$ with ${\diam}\left (\bigcap_{i\in I}C_i\right )=1$,
there exist $s\ls \alpha n$ and $i_1,\ldots i_s\in I$ such that
\begin{align*}
{\diam}(C_{i_1}\cap\cdots \cap C_{i_s})\ls cn^{3/2},
\end{align*}
where $c>0$ are  absolute constant.

\end{theorem}

It should be mentioned that when symmetry is assumed better bounds in both theorems can be obtained (see \cite[Theorem 1.2]{brazitikos2017brascamp} and \cite[Theorem 1.2.]{brazitikos2017quantitative}).

\bigskip

Brazitikos' proofs of Theorem~\ref{teo: volume brazitikos} and Theorem~\ref{th:diameter} rely on the work  of Batson, Spielman and Srivastava  on  approximate  John's decompositions with few vectors \cite{batson2012twice}.
For Theorem~\ref{teo: volume brazitikos},  this is successfully combined with a new and very useful estimate for corresponding `approximate' Brascamp-Lieb-type inequality while, for Theorem~\ref{th:diameter}, the argument is based on a clever lemma of Barvinok from \cite{barvinok2014thrifty}. This lemma in turn, exploits again the theorem of Batson et. al. or to be precise, a more delicate version of Srivastava from \cite{srivastava2012contact}.

\bigskip

Of course if one is willing to further relax the number of intersections involved in the statements of Theorems \ref{teo: volume brazitikos} and \ref{th:diameter}, then one should expect to obtain better bounds/estimates.
The aim of this note is to present  the following \emph{continuous} quantitative Helly-type results (i.e., Helly-type results which are sensitive to the number intersections considered).

\begin{theorem}{(Continuous Helly-type theorem for the volume)}\label{teo: volume main th}
 Let $1 \leq \delta \leq 2$, there is an absolute constant $\alpha>1$ with the following property: for every $n \in \mathbb{N}$ and every family ${\mathcal H}=\{ H_i:i\in I\}$ of closed half-spaces in ${\mathbb R}^n$,
$$
H_i = \{x \in \R^n: \langle x, v_i\rangle \leq 1\},
$$
with $\vol(\bigcap_{i\in I}H_i)=1$, there  exists $s \leq \alpha n^{\delta}$ and $i_1, \dots, i_s \in I$ such that
\begin{align*}
\vol(H_{i_1}\cap \cdots \cap H_{i_s})^{1/n} \leq d_n n^{\frac{3}{2}-\frac{\delta}{2}},
\end{align*}
where $d_n \to 1$ as $n \to \infty$.
\end{theorem}

\begin{theorem}{(Continuous Helly-type theorem for the diameter)}\label{teo: diameter main th}
Let $1 \leq \delta \leq 2$, there is an absolute constant $\alpha>1$ with the following property: 
for every $n \in \mathbb{N}$ and every  finite family  $\{C_i: i\in I\}$ of convex bodies in ${\mathbb R}^n$ such that $\bigcap_{i\in I}C_i$ has non-empty interior, then there is $z\in  \mathbb{R}^n$, a number $s\ls \alpha n^{\delta}$ and indices $i_1,\ldots i_s\in I$ such that

$$ 
z+C_{i_1} \cap \dots C_{i_s} \subset c n^{2-\frac{\delta}{2}} \left( z + \bigcap_{i\in I}C_i \right),
$$

where $c>0$ are  absolute constant.

\bigskip

In particular, for every $n \in \mathbb{N}$ and every  finite family  $\{C_i: i\in I\}$ of convex bodies in ${\mathbb R}^n$ with ${\diam}\left (\bigcap_{i\in I}C_i\right )=1$,
there exist $s\ls \alpha n^{\delta}$ and $i_1,\ldots i_s\in I$ such that
\begin{align*}
{\diam}(C_{i_1}\cap\cdots\cap C_{i_s})\ls cn^{2-\frac{\delta}{2}},
\end{align*}
where $c>0$ are  absolute constant.
\end{theorem}

The possible number of convex sets we allow to intersect on each statement (i.e., $\alpha n^\delta$)  will be referred as \emph{the admissible number of intersections.} 
\bigskip

Note that in both theorems we recover the previous mentioned results when the admissible number of intersections is linear in $n$ (i.e., when $\delta=1$). If the admissible number of intersections is $O(n^2)$ then the bounds are the known ones which follow by directly applying K. Ball classical theorem for the volume ratio \cite{ball1991volume} (see the end of Section~\ref{section: volume} for an argument).
Therefore, the dependencies in the exponent of both results obtained seem to be accurate.
Moreover, for a linear number of spaces (i.e., $\delta=1$) the constant that appears in Theorem~\ref{teo: volume main th} is much better than the one in \cite[Theorem 1.4.]{brazitikos2017brascamp}, since $d_n \to 1$  as $n$ goes to infinity.

\bigskip
To obtain Theorems \ref{teo: volume main th} and \ref{teo: diameter main th} we carefully follow Brazitikos's proofs of Theorems \ref{teo: volume brazitikos} and \ref{th:diameter} 
 but instead of using Batson et. al. or Srivastava's statment on the approximate John's decomposition we replace it with the following stronger result due to 
Friedland and Youssef (who exploited the recent solution of the Kadison-Singer problem \cite{marcus2015interlacing}, by showing that any $n \times m$ matrix $A$ can be approximated
in operator norm by a submatrix with a number of columns of
order the stable rank of $A$).

\bigskip

\begin{theorem}{\cite[Theorem 4.1]{friedland2019approximating}}\label{teo: desc aprox id optima}
Let $\{u_j, a_j\}_{1\leq j\le m}$ be a John's decomposition of the identity i.e, the identity operator $\Id_n$ is decomposed in the form
$\Id_n=\sum_{j=1}^ma_ju_j\otimes u_j$. Then for any $\varepsilon>0$ there exists a multi-set $\sigma \subset [m]$ (i.e., it allows repetitions of the elements) with $|\sigma| \le {n}/{c\varepsilon^2}$ so that 
$$
(1-\varepsilon) \Id_n \preceq \frac{n}{|\sigma|} \sum_{j\in\sigma} (u_j-u)\otimes (u_j-u)\preceq (1 + \varepsilon) \Id_n
$$
where $u=\frac{1}{|\sigma|} \sum_{i\in\sigma} u_j$ satisfies $\|u\|_2\le \frac{2\varepsilon}{3\sqrt{n}}$, and $c>0$ is an absolute constant.
\end{theorem}


In the words of Friedland and Youssef, Thorem \ref{teo: desc aprox id optima} improves  Srivastava’s theorem \cite[Theorem
5]{srivastava2012contact} in three different ways. First the approximation ratio $(1 + \varepsilon)/(1 -
\varepsilon)$ can be made arbitrary close to 1 (while in Srivastava’s result one could
only get a $(4 + \varepsilon)$-approximation). Secondly, it gives an explicit
expression of the weights appearing in the approximation.
Finally, there is a big difference in the
dependence on $\varepsilon$ in the estimate of the norm of $u$: Srivastava obtains a
similar bound but with $\varepsilon$ replaced by 
$\sqrt{\varepsilon}$. 
This behaviour on the $\varepsilon$ parameter will be crucial for our purposes. With this at hand, we take $\varepsilon$ small but \emph{depending explicitly} on the admissible number of intersections (so, it will be linked to the dimension of the ambient space and the $\delta$ parameter).  This seems to be a new idea in these kind of problems.
All this allow us to obtain the estimates on our main results.
\bigskip

As we mentioned before, if all the bodies are symmetric then the known bounds in general are better. In that case, for an admissible number of intersections which is linear in $n$, the bounds are of order $\sqrt{n}$.
For the symmetric case we show in Section~\ref{section: final}  that relaxing the number of admissible intersections does not longer provide stronger estimates. Namely, the exponent $1/2$ cannot be improved. 

\bigskip

To obtain Theorem~\ref{teo: diameter main th} we prove the following proposition which provides  some sort of  ``thrifty'' approximation for a convex body. 
This is a continuous version of \cite[Proposition 4.2.]{brazitikos2017quantitative}, which we feel it may be interesting in its own right. 

\begin{proposition}\label{pro: paso intermedio}
Let $1 \leq \delta \leq2 $. If $K$ is a convex body whose minimal volume ellipsoid is the Euclidean unit ball (i.e., in Löwner position), then there is a subset $X\subseteq K\cap S^{n-1}$ of cardinality $card(X) \leq \alpha n^{\delta}$ and $$K\subseteq B_{2}^{n} \subseteq Cn^{2 - \frac{\delta}{2}} \conv(X),$$
 where $\alpha, C>0$ are absolute constant.
 \end{proposition}

\bigskip
To end with the introduction, we emphasize that there has been a lot of interest in Helly-type results recently. For example, the behavior of volumetric and diametric statements for large subfamilies is
described in \cite{de2017quantitative}. On the other hand, in  \cite{damasdi2021colorful,sarkar2019quantitative} newer exact quantitative Helly-type theorems for the
volume are given. Fractional and colorful versions for the diameter problem are treated in \cite{dillon2020m}.
We refer the reader to the references therein for a broader picture in this area and possible other ramifications.

\section{Notation and background}
We refer to the book of Artstein-Avidan, Giannopoulos and V. Milman \cite{artstein2015asymptotic} for basic facts from convexity and asymptotic geometry.

\smallskip
Recall that a convex body in ${\mathbb R}^n$ is a compact convex subset $K$ of ${\mathbb R}^n$ with non-empty interior. We say that the body $K$ is
symmetric if $x\in K$ implies that $-x\in K$. For any set $X$ we write $\conv(X)$ for its convex hull. For convex body $K$ we write $p_K$ for the Minkowski's functional of $K$, that is
$$p_K(x):= \inf\{\lambda > 0 \;:\; x \in \lambda K\}. $$
If $0\in {\mbox{int}}(K)$ then the polar body $K^{\circ }$ of $K$ is given by
\begin{align*}K^{\circ }:=\{ y\in {\mathbb R}^n: \langle x,y\rangle \ls 1
\;\hbox{for all}\; x\in K\}. \end{align*}
Volume is denoted by $\vol(\cdot)$ and  diameter by $\diam(\cdot).$
We consider in ${\mathbb R}^n$ the  Euclidean structure $\langle\cdot ,\cdot\rangle $ and denote by $\|\cdot \|_2$
the corresponding Euclidean norm. We write $B_2^n$ and $S^{n-1}$ for the corresponding Euclidean unit ball and  unit sphere respectively.

We say that a convex body $K$ is in John's position if the ellipsoid of maximal volume inscribed in $K$ is
the Euclidean unit ball $B_2^n$. John's classical theorem states that
$K$ is in John's position if and only if $B_2^n\subseteq K$ and there exist $u_1,\ldots ,u_m\in {\mbox{bd}}(K)\cap S^{n-1}$ (contact points of $K$ and $B_2^n$) and positive real numbers $a_1,\ldots ,a_m$ such that
\begin{align*}\sum_{j=1}^ma_ju_j=0\end{align*}
and the identity operator $\Id_n$ is decomposed in the form
\begin{equation}\label{eq:decomposition}\Id_n=\sum_{j=1}^ma_ju_j\otimes u_j, \end{equation}
where the rank-one operator $u_j\otimes u_j$ is simply $(u_j\otimes u_j)(y)=\langle u_j,y\rangle u_j$. 



If $u_1,\ldots ,u_m$ are unit vectors that satisfy John's decomposition \eqref{eq:decomposition} with some
positive weights $a_j$. Then, one has the useful equalities
\begin{align*}\sum_{j=1}^ma_j={\tr}(\Id_n)=n\quad \hbox{and}\quad
\sum_{j=1}^ma_j\langle u_j,z\rangle^2=1 \end{align*}
for all $z\in S^{n-1}$. Moreover,
\begin{equation}\label{eq:ball-inside}{\mbox{conv}}\{v_1,\ldots ,v_m\}\supseteq\frac{1}{n}B_2^n.\end{equation}


The body $K$ is in Löwner position if the minimal volume ellipsoid that contains it is the Euclidean ball $B_2^n$. In that case, we also have a decomposition of the identity as before.

Given two matrices $A,B \in \mathbb{R}^{n \times n}$ we write $A \preceq B$ whenever $B-A$ is positive semidefinite. 


The letters $c,c^{\prime }, C,C'$ etc. will always denote absolute positive constants which may change from line to line. 

\section{Continuous Helly-type result for the volume:  Theorem~\ref{teo: volume main th}.} \label{section: volume}

As mentioned above, the proof of Theorem~\ref{teo: volume main th} is based on an earlier one by Brazitikos \cite[Theorem 1.4.]{brazitikos2017brascamp}. 
However, we make two crucial modifications. The first one is that we use stronger results regarding
sparsification of John decompositions of the unity. The second difference is that we exploit how good the approximation of the identity is in terms on the admissible number of intersections. That is, we choose the $\varepsilon$ parameter (which measures the precision of the approximation) appropriately as a function of $n$ and $\delta$.
We include all the steps, even those that are an identical copy of the proof of \cite[Theorem 1.4.]{brazitikos2017brascamp}, for completeness. 

\bigskip
The following Brascamp-Lieb type inequality for approximate John's decomposition of the identity will be crucial.

\begin{theorem}{\cite[Theorem 5.4] {brazitikos2017brascamp}}\label{teo: b-l para desc aprox de la id}
 Let $\gamma>1$. Let $u_1,\cdots,u_s\in{S^{n-1}}$ and $a_1,\cdots,a_s>0$ satisfy $$\Id_n\preceq{A:=\sum_{j=1}^{s}{a_ju_j\otimes{u_j}}}\preceq\gamma{\Id_n}$$ and let $k_j=a_j\langle{A^{-1}u_j,u_j}\rangle>0,\;1\leq{j}\leq{s}$. If $f_1,\cdots,f_s:\R\longrightarrow\R^{+}$ integrable functions then $$\int\limits_{\R^n}{\prod\limits_{j=1}^{s}{f_j^{k_j}\left(\langle{x,u_j}\rangle\right)}}\,\mathrm{d}x\leq\gamma^{\frac{n}{2}}\prod\limits_{j=1}^{s}\left(\int\limits_{\R}{{f_j(t)}\,\mathrm{d}t}\right)^{k_j}.$$
 \end{theorem}
 
 We now prove Theorem~\ref{teo: volume main th}.
 
 \begin{proof}{(of Theorem~\ref{teo: volume main th})}

Without loss of generality we  assume that $P:=\bigcap_{i\in I} H_i$ is in John's position.
Therefore there exist $J\subseteq{I}$ and vectors $(u_j)_{j\in{J}}$ which are contact points between $P$ and $S^{n-1}$ and $(a_j)_{j\in{J}}$ positive numbers, such that $$\Id_n=\sum_{j\in{J}}{a_ju_j\otimes{u_j}}\;\;\;\mbox{and}\;\;\;\sum_{j\in{J}}{a_ju_j}=0.$$

Using  Friedland and Youssef's approximate decomposition, Theorem~\ref{teo: desc aprox id optima}, we can find a multi-set $\sigma\subseteq{J}$ with $|\sigma| \leq \frac{n}{c\varepsilon^2}$ and a vector $u=\frac{-1}{\abs{\sigma}}\sum_{j \in \sigma}{u_j}$ such that $$(1-\varepsilon)\Id_n\preceq\frac{n}{\abs{\sigma}}\sum_{j \in \sigma}{(u_j+u)\otimes(u_j+u)}\preceq(1+\varepsilon)\Id_n,$$
also satisfying that 
  $\frac{n}{\abs{\sigma}}\sum_{j \in \sigma}{(u_j+u)}=0$ and $\Vert{u}\Vert_2\leq\frac{2\varepsilon}{3\sqrt{n}}$. 
 
  We consider the vector $w:=\frac{3u}{2\sqrt{n}\varepsilon}$.  Recall that $\frac{1}{n}B_2^n\subseteq\conv\{u_j,\;j\in{J}\}$, thus $\Vert w \Vert_2\leq\frac{1}{n}$ and hence $w\in\conv\{u_j,\;j\in{J}\}$.
By Carathéodory's Theorem, we know that there is $\tau\subseteq{J}$, with $\abs{\tau}\leq{n+1}$ and $\rho_i>0,\;i\in\tau$ such that 
 $$w=\sum_{i\in\tau}{\rho_iu_i}\;\;\;\mbox{and}\;\;\;\sum_{i\in\tau}{\rho_i}=1.$$
 
 Also notice that, since
   $u=\frac{-1}{\abs{\sigma}}\sum_{j \in \sigma}{u_j}$ and $\sum_{j \in \sigma}{\frac{1}{\abs{\sigma}}}=1$,  $-u\in\conv\{u_j,\;j \in \sigma\}$. Therefore, we have that the segment $\left[-u,\frac{3u}{2\sqrt{n}\varepsilon}\right]$ is contained in $\conv\{u_j,\;j\in{\sigma\cup\tau}\}.$ For $j \in \sigma$ we define $$v_j:=\sqrt{\frac{n}{n+1}}\left(-u_j,\frac{1}{\sqrt{n}}\right)\;\;\;\mbox{and}\;\;\;b_j=\frac{n+1}{\abs{\sigma}}.$$
 Set $v:=-\sqrt{\frac{n}{n+1}}\left(u,0\right)$. So, we have
 \begin{align*}
 \sum_{j \in \sigma}{b_j(v_j+v)\otimes(v_j+v)}&=\sum_{j \in \sigma}{\frac{n}{\abs{\sigma}}\left(-(u_j+u),\frac{1}{\sqrt{n}}\right)\otimes\left(-(u_j+u),\frac{1}{\sqrt{n}}\right)}\\
 &=
\left(\begin{matrix}
        \sum_{j \in \sigma}{\frac{n}{\abs{\sigma}}(u_j+u)\otimes(u_j+u)}&\frac{\sqrt{n}}{\abs{\sigma}}\sum_{j \in \sigma}{(u_j+u)}\\
        \frac{\sqrt{n}}{\abs{\sigma}}\sum_{j \in \sigma}{(u_j+u)^t}&\frac{n}{\abs{\sigma}}\sum_{j \in \sigma}{\frac{1}{n}}
      \end{matrix}\right)
\\
&=
\left(\begin{matrix}
        \sum_{j \in \sigma}{\frac{n}{\abs{\sigma}}(u_j+u)\otimes(u_j+u)}&0\\
        0&1
      \end{matrix}\right)
,
 \end{align*}
 which implies 
 \begin{equation}\label{eq: aprox de id n+1 para teo 6.4}
 (1-\varepsilon)\Id_{n+1}\preceq\sum_{j \in \sigma}{b_j(v_j+v)\otimes(v_j+v)}\preceq(1+\varepsilon)\Id_{n+1}.
 \end{equation}
 The sum $\sum_{j \in \sigma}{b_j(v_j+v)\otimes(v_j+v)}$ can be written as $$\sum_{j \in \sigma}{b_jv_j\otimes{v_j}}+v\otimes{\left(\sum_{j \in \sigma}{b_jv_j}\right)}+\left(\sum_{j \in \sigma}{b_jv_j}\right)\otimes{v}+(n+1)v\otimes{v},$$ and notice that since
 \begin{align*}
{}\sum_{j \in \sigma}{b_jv_j}&=\sum_{j \in \sigma}{\frac{n+1}{\abs{\sigma}}\sqrt{\frac{n}{n+1}}\left(-u_j,\frac{1}{\sqrt{n}}\right)}\\
&=\sqrt{\frac{n+1}{n}}\left(-\sum_{j \in \sigma}\frac{n}{\abs{\sigma}}{u_j},\frac{1}{\abs{\sigma}}\sum_{j \in \sigma}{\sqrt{n}}\right)\\
 &=\sqrt{\frac{n+1}{n}}\left(nu,\sqrt{n}\right),
 \end{align*}
we obtain that $$\left(\sum_{j \in \sigma}{b_jv_j}\right)\otimes{v}=\sqrt{\frac{n+1}{n}}\left(nu,\sqrt{n}\right)\otimes{\sqrt{\frac{n}{n+1}}(-u,0)}=
\left(\begin{matrix}
        -nu\otimes{u}&0\\
        -\sqrt{n}u^t&0
      \end{matrix}\right)
,$$
$$v\otimes{\left(\sum_{j \in \sigma}{b_jv_j}\right)}=
\left(\begin{matrix}
        -nu\otimes{u}&-\sqrt{n}u\\
        0&0
      \end{matrix}\right)
,$$
and $(n+1)v\otimes{v}=
\left(\begin{matrix}
        nu\otimes{u}&0\\
        0&0
      \end{matrix}\right)
$. 

Hence, we can write Equation \eqref{eq: aprox de id n+1 para teo 6.4} as $$(1-\varepsilon)\Id_{n+1}-T\preceq\sum_{j \in  \sigma}{b_jv_j\otimes{v_j}}\preceq(1+\varepsilon)\Id_{n+1}-T,$$
where $T=v\otimes{\left(\sum_{j \in \sigma}{b_jv_j}\right)}+\left(\sum_{j \in \sigma}{b_jv_j}\right)\otimes{v}+(n+1)v\otimes{v}=
\left(\begin{matrix}
        V&z\\
        z^t&0
      \end{matrix}\right)
,$ with $V=-nu\otimes{u}$ y $z=-\sqrt{n}u$. Now,  for $(x,t)\in{S^{n}}$ where $x \in \mathbb{R}^n$ and $t$ is a scalar, we have that 
\begin{align*}
{}\left\langle{T(x,t),(x,t)}\right\rangle&=\left\langle{\left(Vx+zt,\langle{z,x}\rangle\right),(x,t)}\right\rangle\\
&=\left\langle{\left(Vx,0\right),(x,t)}\right\rangle+\left\langle{(zt,\langle{z,x}\rangle),(x,t)}\right\rangle\\
&\leq\langle{Vx,x}\rangle+\Vert{(zt,\langle{z,x}\rangle)}\Vert_2\Vert{(t,x)}\Vert_2=\langle{Vx,x}\rangle+\left(\Vert{zt}\Vert_2^2+\langle{z,x}\rangle^2\right)^{\frac{1}{2}}\\
&\leq\norm{V}\Vert{x}\Vert_2^2+\left(\Vert{z}\Vert_2^2t^2+\Vert{z}\Vert_2^2\Vert{x}\Vert_2^2\right)^{\frac{1}{2}}\leq\norm{V}+\Vert{z}\Vert_2\left(t^2+\Vert{x}\Vert_2^2\right)^{\frac{1}{2}}\\
&=\norm{V}+\Vert{z}\Vert_2\Vert{(x,t)}\Vert_2=\norm{V}+\Vert{z}\Vert_2=n\Vert{u}\Vert_2^2+\sqrt{n}\Vert{u}\Vert_2\\
&\leq{n\frac{4\varepsilon^2}{9n}+\sqrt{n}\frac{2\varepsilon}{3\sqrt{n}}}=\frac{4\varepsilon^2}{9}+\frac{2\varepsilon}{3}\\
&\leq\varepsilon,
\end{align*}
for $\varepsilon$ small enough (say $\varepsilon\leq\frac{3}{4}$). So, $\norm{T}\leq\varepsilon$, and hence Equation \eqref{eq: aprox de id n+1 para teo 6.4} implies that $$(1-2\varepsilon)\Id_{n+1}\preceq{A:=\sum_{j \in \sigma}{b_jv_j\otimes{v_j}}}\preceq(1+2\varepsilon)\Id_{n+1},$$
or equivalently $$\Id_{n+1}\preceq\sum_{j \in \sigma}{\frac{b_j}{1-2\varepsilon}v_j\otimes{v_j}}\preceq\gamma{\Id_{n+1}},$$ with $\gamma=\frac{1+2\varepsilon}{1-2\varepsilon}$. Applying Theorem~\ref{teo: b-l para desc aprox de la id},  if $f_j:\R\rightarrow\R^{+}$ are measurable functions, then $$\int\limits_{\R^{n+1}}{\prod\limits_{j \in \sigma}{f_j^{k_j}\left(\langle{x,v_j}\rangle\right)}}\,\mathrm{d}x\leq\gamma^{\frac{n+1}{2}}\prod\limits_{j \in \sigma}\left(\int\limits_{\R}{{f_j(t)}\,\mathrm{d}t}\right)^{k_j},$$
where $$k_j=\frac{b_j}{1-2\varepsilon}\left\langle{\left(\frac{1}{1-2\varepsilon}A\right)^{-1}v_j,v_j}\right\rangle=b_j\langle{A^{-1}v_j,v_j}\rangle.$$ Since $A^{-1}\preceq\frac{1}{1-2\varepsilon}\Id_{n+1}$,  we have that $\frac{k_j}{b_j}\leq\frac{1}{1-2\varepsilon}$.
Note also that $\sum_{j \in \sigma} k_j = n+1$. Indeed,
 
\begin{align*} 
\sum_{j\in  \sigma} \kappa_j &= \sum_{j\in  \sigma} b_j \langle A^{-1}v_j,v_j \rangle =\sum_{j\in  \sigma} b_j \,{\tr}(v_j\otimes A^{-1}v_j)={\tr}\left (\sum_{j\in  \sigma} b_j(v_j\otimes A^{-1}v_j)\right )\\
 &={\tr}\left (\sum_{j\in  \sigma}b_j A^{-1}(v_j\otimes v_j)\right )= {\tr}\left (A^{-1}\Big (\sum_{j\in  \sigma} b_j(v_j\otimes v_j)\Big )\right ) \\
& ={\tr}(A^{-1}A)={\tr}(\Id_{n+1})=n+1.
\end{align*}

Now for $j \in \sigma$ we consider $f_j(t):=e^{\frac{-b_j}{k_j}t}\chi_{[0,\infty)}(t)$. So,
\begin{align*}
\int\limits_{\R^{n+1}}{\prod\limits_{j \in \sigma}{f_j^{k_j}\left(\langle{x,v_j}\rangle\right)}}\,\mathrm{d}x&\leq\gamma^{\frac{n+1}{2}}\prod\limits_{j \in \sigma}\left(\int\limits_{\R}{{f_j(t)}\,\mathrm{d}t}\right)^{k_j}\\
&=\gamma^{\frac{n+1}{2}}\prod\limits_{j \in \sigma}{\frac{k_j}{b_j}}^{k_j}\\
&\leq\gamma^{\frac{n+1}{2}}\frac{1}{(1-2\varepsilon)^{\sum_{j \in \sigma}{k_j}}}=\gamma^{\frac{n+1}{2}}\frac{1}{(1-2\varepsilon)^{n+1}}\\
&=\left(\frac{1+2\varepsilon}{(1-2\varepsilon)^3}\right)^{\frac{n+1}{2}}.
\end{align*}
Set $$Q=\bigcap_{i\in{\sigma\cup\tau}}{H_i}=\{x\in\R^n:\;\langle{x,u_j}\rangle<1,\;j\in{\sigma\cup\tau}\},$$
and let $y=(x,r)\in\R^{n+1}$. Assume that $r>0$ and $x\in{\frac{r}{\sqrt{n}}Q}$. Then we have that $\langle{x,u_j}\rangle<\frac{r}{\sqrt{n}}$ for every $j \in \sigma$, which implies that $\langle{y,v_j}\rangle>0$ for every $j \in \sigma$, and then $\prod\limits_{j \in \sigma}{f_j^{k_j}(\langle{y,v_j}\rangle)}>0$. We also have that 
\begin{align*}
\left\langle{\frac{1}{\abs{\sigma}}\sum_{j \in \sigma}{u_j},x}\right\rangle&=\langle{-u,x}\rangle=\frac{2\sqrt{n}\varepsilon}{3}\langle{-w,x}\rangle\\
&=\frac{2\sqrt{n}\varepsilon}{3}\left\langle{-\sum_{i\in\tau}{\rho_iu_i},x}\right\rangle\geq\frac{-2\sqrt{n}\varepsilon}{3}\left(\sum_{i\in\tau}{\rho_i}\right)\frac{r}{\sqrt{n}}\\
&=\frac{-2\varepsilon{r}}{3}.
\end{align*}
Thus, if $y=(x,r)\in\frac{r}{\sqrt{n}}Q\times(0,\infty)$, then 
\begin{align*}
\prod\limits_{j \in \sigma}{f_j^{k_j}(\langle{y,v_j}\rangle)}&=exp\left(-\sum_{j \in \sigma}{b_j\left(\frac{r}{\sqrt{n+1}}-\sqrt{\frac{n}{n+1}}\langle{x,u_j}\rangle\right)}\right)\\
&=exp\left(\frac{-r}{\sqrt{n+1}}\sum_{j \in \sigma}{b_j}\right)exp\left(\sqrt{n}\sqrt{n+1}\left\langle{x,\frac{1}{\abs{\sigma}}\sum_{j \in \sigma}{u_j}}\right\rangle\right)\\
&\geq{e^{-r\sqrt{n+1}}e^{-\sqrt{n}\sqrt{n+1}\frac{2}{3}r\varepsilon}}=e^{-r\sqrt{n+1}\left(1+\frac{2}{3}\varepsilon\sqrt{n}\right)}.
\end{align*}
Now, by Theorem~\ref{teo: b-l para desc aprox de la id},
\begin{align*}
\frac{\vol(Q)}{n^{\frac{n}{2}}}\int\limits_{0}^{\infty}{r^ne^{-r\sqrt{n+1}\left(1+\frac{2}{3}\varepsilon\sqrt{n}\right)}}\,\mathrm{d}r&=\int\limits_{0}^{\infty}\int\limits_{\frac{r}{\sqrt{n}}Q}{e^{-r\sqrt{n+1}\left(1+\frac{2}{3}\varepsilon\sqrt{n}\right)}}\,\mathrm{d}x\,\mathrm{d}r\\
&\leq\int\limits_{\R^{n+1}}{\prod\limits_{j \in \sigma}{f_j^{k_j}(\langle{y,v_j}\rangle)}}\,\mathrm{d}y\\
&\leq\left(\frac{1+2\varepsilon}{(1-2\varepsilon)^3}\right)^{\frac{n+1}{2}}.
\end{align*}
Using that $B_2^n\subseteq{P}$, and the fact that $$\int\limits_{0}^{\infty}{r^ne^{-r\sqrt{n+1}\left(1+\frac{2}{3}\varepsilon\sqrt{n}\right)}}\,\mathrm{d}r=\frac{n!}{(n+1)^{\frac{n+1}{2}}\left(1+\frac{2}{3}\varepsilon\sqrt{n}\right)^{n+1}},$$
we obtain, by taking $1+\varepsilon'=\frac{1+2\varepsilon}{(1-2\varepsilon)^3}$,  
\begin{align*}
{}\vol\big(\bigcap_{i\in{\sigma\cup\tau}}{H_i} \big)=\vol(Q)&\leq\frac{(1+\varepsilon')^{\frac{n+1}{2}}n^{\frac{n}{2}}(n+1)^{\frac{n+1}{2}}\left(1+\frac{2}{3}\varepsilon\sqrt{n}\right)^{n+1}}{n!}\frac{\vol(P)}{\vol(B_2^n)}\\
&=\frac{(1+\varepsilon')^{\frac{n+1}{2}}n^{\frac{n}{2}}(n+1)^{\frac{n+1}{2}}\left(1+\frac{2}{3}\varepsilon\sqrt{n}\right)^{n+1}}{n!}\frac{\Gamma\left(\frac{n}{2}+1\right)\vol(P)}{\pi^{\frac{n}{2}}}.
\end{align*}
By Stirling's formula we get, for a constant $C>0$, the inequality 
\begin{align*}
{}\vol({\bigcap_{i\in{\sigma\cup\tau}}{H_i}})&\leq{C\frac{(1+\varepsilon')^{\frac{n+1}{2}}n^{\frac{n}{2}}(n+1)^{\frac{n+1}{2}}\left(1+\frac{2}{3}\varepsilon\sqrt{n}\right)^{n+1}}{\pi^{\frac{n}{2}}\sqrt{2\pi{n}}\left(\frac{n}{e}\right)^n}\frac{n}{2}\sqrt{\frac{4\pi}{n}}\left(\frac{n}{2e}\right)^{\frac{n}{2}}\vol(P)}\\
&=C(1+\varepsilon')^{\frac{n+1}{2}}\frac{\left(1+\frac{2}{3}\varepsilon\sqrt{n}\right)^{n+1}n^nn(n+1)^{\frac{n+1}{2}}}{n^nn}\left(\frac{e}{2\pi}\right)^{\frac{n}{2}}\frac{1}{\sqrt{2}}\vol(P)\\
&=C(1+\varepsilon')^{\frac{n+1}{2}}\left(1+\frac{2}{3}\varepsilon\sqrt{n}\right)^{n+1}(n+1)^{\frac{n+1}{2}}\left(\frac{e}{2\pi}\right)^{\frac{n}{2}}\frac{1}{\sqrt{2}}\vol(P).
\end{align*}

We now define the accuracy of the approximation in terms of the admissible number of intersections.
Fix $\varepsilon:= \frac{1}{4} n^{(1-\delta)/2}$, using that $1+\varepsilon'=\frac{1+2\varepsilon}{(1-2\varepsilon)^3}$ we have  $$\left(1+\varepsilon{'}\right)\left(1+\frac{2}{3}\varepsilon\sqrt{n}\right)^2\frac{e}{2\pi}=(1+\varepsilon{'})\left(1+\frac{1}{6}n^{(2-\delta)/2}\right)^2\frac{e}{2\pi}<cn^{2-\delta}.$$  

Therefore,
\begin{align*}
{}\vol({\bigcap_{i\in{\sigma\cup\tau}}{H_i}})&\leq{C n^{\frac{n}{2}}\left(1+\frac{1}{n}\right)^{\frac{n}{2}}\sqrt{n+1}\;n^{(2-\delta)/2}\;n^{n(2-\delta)/2}\;\frac{\sqrt{2\pi}}{\sqrt{2e}}\vol(P)}\\
&\leq{C_1\sqrt{\frac{e(n+1)\pi}{e}}\;n^{\frac{n}{2}}\;n^{(2-\delta)/2}\;n^{n(2-\delta)/2}\vol(P)}\\
&=\left(\underbrace{C_1\sqrt{n+1}n^{(2-\delta)/2}\sqrt{\pi}}_{C_n}\right)n^{n(3-\delta)/2}\vol(P).
\end{align*}
We conclude that $$\vol({\bigcap_{i\in{\sigma\cup\tau}}{H_i}})\leq{C_n\;n^{n(3-\delta)/2}\;\vol(P)},$$
where the intersection is taken over at most  $\abs{\sigma\cup\tau}\leq\frac{n}{c\varepsilon^2}+n+1=\frac{n^{\delta}}{c}+n+1 \leq \alpha n^{\delta}$ half-spaces.
Since the constant $C_n$ is of order $n^{(3-\delta)/2}$, we have that $d_n:=C_n^{1/n} \to 1$ as $n \to \infty$. 

\end{proof}
It should be mentioned that the case $\delta=2$ is easier since it follows directly from very known results in the literature (which also give a better constant). Indeed, assume that $P$ is John's position and let $H_{i_1}, \dots, H_{i_s}$ be the (at most $O(n^2)$) hyperplanes determined by the contact points between $P$ and the Euclidean unit ball $B_2^n$ which provide a decomposition of the identity. Note that $\bigcap H_{i_s}$ is itself in John's position. Thus, by the estimate on the volume ratio of K. Ball \cite[Theorem 1']{ball1991volume}  we have
$$ \frac{\vol\left(\bigcap_{i=1}^s H_{i_s}\right)}{\vol(B_2^n)} \leq \frac{\vol\left(\Delta\right)}{\vol(B_2^n)} 
=   \frac{n^{\frac{n}{2}} (n+1)^{\frac{n+1}{2}} \Gamma(\frac{n}{2}+1)}{n! \pi^{\frac{n}{2}}},$$
where $\Delta$ stands for the regular simplex in John's position.
Thus, using Stirling formula and the fact that $P\supset B_2^n$,
$ \vol\left(\bigcap_{i=1}^s H_{i_s}\right) \leq c^n n^{n/2} \vol(P)$ where 
$c>0$ is an absolute constant strictly smaller than one.
We should mention that the Brascamp-Lieb inequality is of course hidden in the proof of Ball's result.

 \section{Continuous Helly-type theorem for the diameter}

We start with the proof of Proposition~\ref{pro: paso intermedio} which be need to prove Theorem~\ref{teo: diameter main th}. Although the argument is similar to \cite[Proposition 4.2.]{brazitikos2017quantitative}  we include all the steps for completeness. 
Again the improvement is to have a \emph{good} approximation of the identity and a nice control of its accuracy.

\begin{proof}[Proof of Proposition~\ref{pro: paso intermedio}]
By John's theorem there exist $v_{j}\in K\cap S^{n-1}$ and $a_{j}>0$, $j\in J$, such that 
 $$\Id_{n}= \sum_{j\in J} {a_{j} v_{j}\otimes v_{j}} \; \mbox{and} \; \sum_{j\in J} {a_{j}v_{j}}=0.$$
 Let $\varepsilon>0$ small enough to be fixed later. By Theorem~\ref{teo: desc aprox id optima} we can find a multiset $\sigma \subseteq J$ of cardinal $\vert \sigma \vert \leq \frac{n}{c \varepsilon^2}$ such that 
 $$(1-\varepsilon)\Id_{n}\preceq \frac{n}{\sigma} \sum_{j \in \sigma}{(v_{j}+v)\otimes (v_{j}+v)} \preceq (1+\varepsilon)\Id_{n},$$
 where $v=\frac{-1}{\vert\sigma\vert}\sum_{j \in \sigma}{v_{j}}$ satisfies $\Vert v \Vert_2 \leq \frac{2\varepsilon}{3\sqrt{n}}$.
 
 Then, the vector $w=\frac{3v}{2\sqrt{n}\varepsilon}$ satisfies $\Vert w \Vert_2 \leq \frac{1}{n}$ and therefore by Equation \eqref{eq:ball-inside}, it belongs to $\conv\{v_{j}:\;j\in J\}$.  By Carathéodory's theorem there exist $\tau\subseteq J$ with $\vert \tau \vert \leq n+1$ and $\rho_{i}>0$, $i\in \tau$ such that
$$w=\sum_{i\in\tau}{\rho_{i}v_{i}},\;\mbox{and}\;\sum_{i\in\tau}{\rho_{i}}=1.$$
Observe also that $-v=\frac{1}{\vert \sigma \vert}\sum_{j \in \sigma}{v_{j}}$ is in $\conv\{v_{j}:\; j \in \sigma\}$. Let 
$$T:= \frac{n}{\vert \sigma\vert}\sum_{j \in \sigma}{v_{j}\otimes v} + \frac{n}{\vert\sigma\vert}\sum_{j \in \sigma}{v\otimes v_{j}} + v\otimes v.$$

As in the proof of Theorem~\ref{teo: volume main th} it is easy to see that $\vert \langle{Tx,x}\rangle\vert\leq \varepsilon$ for every unit vector $x \in \R^n$ (provided that $\varepsilon$ is small enough).
Thus
$$(1-2\varepsilon)\Id_{n}\preceq(1-\varepsilon)\Id_{n}-T \preceq \frac{n}{\vert\sigma\vert}\sum_{j \in \sigma} {v_{j}\otimes v_{j}} \preceq (1+\varepsilon)\Id_{n}-T \preceq (1+2\varepsilon)\Id_{n}.$$

Define $X:=\{v_j : j \in \sigma \cup \tau \}$ and $E:=\conv(X).$
Let us  show that $B_{2}^{n} \subseteq c \varepsilon n^{3/2} E$.
Indeed, let $x\in S^{n-1}$; set $A:=\frac{n}{\vert \sigma \vert}\sum_{j \in \sigma}{ v_{j}\otimes v_{j}}$ and $\rho :=\min\{ \langle{x,v_{j}}\rangle:\; j \in \sigma\}$. Note that $\vert \rho \vert \leq 1$ and $\langle{x,v{j}}\rangle - \rho \leq 2$ for all $j \in \sigma$. 

If $\rho<0$ we have
\begin{align*}
{} p_{E}(Ax)&\leq p_{E}\left(Ax- \rho \frac{n}{\vert \sigma \vert}\sum_{j \in \sigma}{v_{j}}\right) + p_{E}\left(\rho \frac{n}{\vert \sigma \vert}\sum_{j \in \sigma}{v_{j}}\right)\\
&=p_{E}\left(\sum_{j \in \sigma}{ \frac{n}{\vert \sigma \vert} (\langle{x,v_{j}}\rangle - \rho) v_{j}} \right) + p_{E}(n \rho(-v))\\
&\leq \sum_{j \in \sigma}{ \frac{n}{\vert \sigma \vert} (\langle{x,v_{j}}\rangle - \rho) p_{E}(v_{j})} - n \rho p_{E}(v) \\
&\leq n \left( 2 + \frac{2\sqrt{n}\varepsilon}{3}p_{E}(w) \right)\\
&\leq c_{1} \varepsilon n^{3/2},
\end{align*}
where we are using that $w\in K$ and therefore $p_{E}(w)\leq 1$. 

On the other hand, if $\rho \geq 0$, then $\langle{x,v_{j}}\rangle\geq 0$ for all $j\in \sigma$, therefore
$$p_{E}(Ax)=p_{E}\left(\frac{n}{\vert\sigma\vert}\sum_{j \in \sigma}{\langle{x,v_{j}}\rangle v_{j}}\right)\leq \frac{n}{\vert \sigma\vert}\sum_{j \in \sigma}{\langle{x,v_{j}}\rangle p_{E}(v_{j}})\leq n.$$
This says that $$p_{A^{-1}(E)}(x)\leq c_{2} \varepsilon n^{3/2}$$
for all $x\in S^{n-1}$, where $c_{2}>0$ is an absolute constant. 

Therefore we have $$(1-2\varepsilon)B_{2}^{n} \subseteq A(B_{2}^{n}) \subseteq c_{2} (1+2\varepsilon) \varepsilon n^{3/2} E.$$
Finally, fix $\varepsilon:=\frac{1}{4} n^{\frac{1}{2}-\frac{\delta}{2}}$. Since $K$ is in  Löwner's position $$K\subseteq B_{2}^{n} \subseteq C_{2} \frac{1+2\varepsilon}{1-2\varepsilon} \varepsilon n^{3/2} \subseteq C n^{2- \frac{\delta}{2}} \conv(X),$$
with $\vert X \vert = \vert \sigma \cup \tau \vert \leq c n^{\delta} + n +1\leq \alpha n^{\delta}$.
 \end{proof}
 Let us now see the proof of the Theorem~\ref{teo: diameter main th}.
 \begin{proof}
 Consider $P:=\bigcap_{i\in I}{C_{i}}$.  Without loss of generality we can assume that $0\in int(P)$ and that the polar body $$P^{\circ}= \conv(\bigcap_{i\in I}{C_{i}^{\circ}})$$ is in Löwner's position.  Using Proposition~\ref{pro: paso intermedio} for the body $K=P^{\circ}$, we know there exists a set $X=\{v_{1},\cdots,v_{s}\}\subseteq P^{\circ}\cap S^{n-1}$ such that $\vert X \vert \leq \alpha n^{\delta}$ and $$P^{\circ} \subseteq C n^{2 - \frac{\delta}{2}} \conv(X),$$
 where $C>0$ is an absolute constant. Since $v_{1},\cdots,v_{s}$ are contact points between $P^{\circ}$ and $B_{2}^{n}$, then we have that $v_{j}\in \bigcap_{i\in I} {C_{i}^{\circ}}$ for all $j=1,\cdots,s$. This implies that there exist $s \leq \alpha n^{\delta}$ and  bodies $\{C_{i_{j}}\}$, such that $v_{j} \in C_{i_{j}}^{\circ}$ for all $j=1\cdots,s$. Then $\conv(X) \subseteq \conv(C_{i_1}^{\circ}\cup\cdots\cup C_{i_{s}}^{\circ})$ and hence $$P^{\circ}\subseteq C n^{2 - \frac{\delta}{2}} \conv(C_{i_1}^{\circ}\cup\cdots\cup C_{i_{s}}^{\circ}).$$
 This shows that $$C_{i_{1}}\cap\cdots\cap C_{i_{s}}\subseteq c n^{2 - \frac{\delta}{2}} P.$$ 
 In particular if $\diam(P)=1$ we have the following estimate for the diameter $$\diam(C_{i_{1}}\cap\cdots\cap C_{i_{s}}) \leq c n^{2- \frac{\delta}{2}}.$$ This concludes the proof.

 \end{proof}
 \section{Final comments: symmetry assumption} \label{section: final}
As mentioned in the introduction, it is known that if all the bodies are symmetric then the bounds for Helly-type results are better (see, for example, \cite[Theorem 1.2]{brazitikos2017brascamp} and \cite[Theorem 1.2.]{brazitikos2017quantitative}). In that case, for a linear number of  intersections, the bounds are of order $\sqrt{n}$.
One should be tempted to think that  increasing the admissible number of intersections would provide stronger estimates but, unfortunately, this is not the case: the \emph{exponent} $1/2$ cannot be improved, as the following example shows.

\begin{remark}
 There is a family ${\mathcal H}=\{ H_i:i\in I\}$ of closed symmetric strips in ${\mathbb R}^n$,
$$
H_i = \{x \in \R^n: \vert \langle x, v_i\rangle \vert \leq 1\},
$$
with $\vol(\bigcap_{i\in I}H_i)=1$ with the following property: for any $\delta \geq 1$ and any collection of indices $i_1, \dots, i_s \in I$ with $s=n^{\delta}$ we have
\begin{align*}
\vol(H_{i_1}\cap \cdots \cap H_{i_s})^{1/n} \geq c(\delta) \frac{\sqrt{n}}{\log(1+n)},
\end{align*}
where $c(\delta)>0$ is a constant which depends exclusively on $\delta$. 

\bigskip
In particular, if $s=n^{\delta}$ and $i_1, \dots, i_s \in I$ is such that we have the following inclusion
$$
H_{i_1}\cap \cdots \cap H_{i_s} \subset \beta \bigcap_{i\in I}H_i,
$$
then $$\beta \geq c(\delta) \frac{\sqrt{n}}{\log(1+n)}.$$
\end{remark}

\begin{proof} Let $I$ be a finite set of indices  and $w_i \in S^{n-1}$ for $i \in I$  such that
\begin{equation}B_2^n\subseteq \bigcap_{i \in I} \hat H_i\subseteq 2B_2^n,\end{equation}
where $\hat H_i$ is defined as the symmetric strip
\begin{equation*}\hat H_i=\{ x\in {\mathbb R}^n:|\langle x,w_i\rangle|\ls 1\}.
\end{equation*}
The existence of $\bigcap_{i \in I}$ can be derived, for example, using duality and a probabilistic argument as in the proof of \cite[Lemma 3.1.]{alonso2008isotropy} (e.g., picking $I=\{1, \dots, e^{cn}\}$ for $c>0$ large enough).

Denote by $Q:= \bigcap_{i \in i} \hat H_i\subseteq 2B_2^n$. 
Now if $s=n^{\delta}$, then for \emph{any} choice of indices $i_1,\ldots ,i_s\in I$ we can use the 
classical lower bound for the volume of the intersection of strips due
to B{\'a}r{\'a}ny-F{\"u}redi \cite{barany1988approximation} (see also known as Carl-Pajor \cite{carl1988gelfand} or Gluskin \cite{gluskin1989extremal}), which 
shows that
\begin{equation}\label{lower}|\hat H_{i_1}\cap\cdots \cap \hat H_{i_s}|^{1/n}\gr \frac{C(\delta)}{\sqrt{\log (1+ n)}}.
\end{equation}

Therefore, if we define for $i \in I$ the symmetric strip $H_i:=\frac{1}{\vol(Q)^{1/n}} \hat H_i$, since $\vol(Q)^{1/n} \sim \frac{1}{\sqrt{n}}$, then the family $\{ H_i : i \in I\}$  satisfies the assertion.
\end{proof}
 
\section*{Acknowledgements}
The second author wants to thank A. Giannopoulos 
who  encourage the writing of this manuscript.
\smallskip

This research was supported by  ANPCyT-PICT-2018-04250 and CONICET-PIP 11220130100329CO.
\bibliographystyle{alpha}

\end{document}